\documentclass[reqno,a4paper]{amsart}
\usepackage{amsmath}
\usepackage{amssymb,amsfonts}
\usepackage{amsthm}
\usepackage{enumerate}
\usepackage[mathscr]{eucal}
\usepackage{eqlist}
\usepackage{cite}

\theoremstyle{plain}
\newtheorem{theorem}{Theorem}[section]
\newtheorem{prop}[theorem]{Proposition}
\newtheorem{lemma}{Lemma}[section]
\newtheorem{corol}{Corollary}[theorem]

\theoremstyle{definition}

\theoremstyle{remark}

\numberwithin{equation}{section}

\title{Resolution of the equation $(3^{x_1}-1)(3^{x_2}-1)=(5^{y_1}-1)(5^{y_2}-1)$}
\thanks{This work has been made in the frame of the ``{\sc Efop}-3.6.1-16-2016-00018 - Improving the role of the research $+$ development $+$ innovation in the higher education through institutional developments assisting intelligent specialization in Sopron and Szombathely''.}

\author{K\'alm\'an Liptai}
   \address{Institute of Mathematics and Informatics, Eszterh\'azy K\'aroly University, Eger, Hungary}
   \email{liptai.kalman@uni-eszterhazy.hu}

\author{L\'aszl\'o N\'emeth}
   \address{Institute of Mathematics, University of Sopron, Sopron, Hungary}
   \email{nemeth.laszlo@uni-sopron.hu}

\author{G\"okhan Soydan}
   \address{Department of Mathematics, Bursa Uluda\u{g} University, Bursa, Turkey}
   \email{gsoydan@uludag.edu.tr}

\author{L\'aszl\'o Szalay}
   \address{Department of Mathematics and Informatics, Jan Selye University, Kom\'arno, Slovakia}
   \email{szalay.laszlo@uni-sopron.hu}
   \thanks{Supported by Hungarian National Foundation for Scientific Research Grant No.~128088.}

 \date{}

  \keywords{Exponential diophantine equation, linear recurrence, Baker method.}
  \subjclass{11D61, 11B37}

\begin{document}

  \begin{abstract}
 Consider the diophantine equation $(3^{x_1}-1)(3^{x_2}-1)=(5^{y_1}-1)(5^{y_2}-1)$ in positive integers $x_1\le x_2$, and $y_1\le y_2$. Each side of the equation is a product of two terms of a given binary recurrence, respectively. In this paper, we prove that the only solution to the title equation is $(x_1,x_2,y_1,y_2)=(1,2,1,1)$. The main novelty of our result is that we allow products of two terms on both sides.
  \end{abstract}

\maketitle

\section{Introduction}
This paper is devoted to investigation of positive integers having a specific but analogous structure of digits in two distinct integer bases. We  worked out the details only for the bases 3 and 5, but our approach and arguments should work for other bases, as well. More precisely, we determine the solutions to the diophantine equation
\begin{equation}\label{main}
(3^{x_1}-1)(3^{x_2}-1)=(5^{y_1}-1)(5^{y_2}-1)
\end{equation}
in positive integers $x_1\le x_2$ and $y_1\le y_2$.

Senge and Strauss \cite{SS} proved that the number of integers for which the sum of digits simultaneously in base $a$ and $b$ do not exceed a given bound is finite if and only if $(\log a)/(\log b)$ not rational. Their method is not effective, and it motivated Stewart \cite{S} to exhibit a lower bound for the sum of the digits in base $a$ and $b$. 
To be precise he proved the following theorem. Assume that $a,b,n\in\mathbb{N}\setminus\{0,1\}$, $\alpha,\beta\in\mathbb{N}$ with $\alpha<a$ and $\beta<b$. If $N(\alpha,a)$ denotes the number of digits different from $\alpha$ in the canonical expansion of $n$ in base $a$ ($N(\beta,b)$ analogously), then
$$
N(\alpha,a)+N(\beta,b)>\frac{\log\log n}{\log\log\log n+C}-1,\qquad (n>25)
$$
provided by $(\log a)/(\log b)$ is irrational. Here $C$ is an effectively computable positive real number depending on $a$ and $b$ only.

This result is followed by several papers on Diophantine equations concerning multi-base representation  of integers, see, for example \cite{BHLS} and the references therein. 

Both sides of equation (\ref{main}) can also be considered as the product of two terms of a binary recurrence, respectively. More generally, but only with one term on both sides Schlickewei and Schmidt \cite{SS1} characterized all the pairs of recurrences $(G, H)$ having infinitely many solution to $G_x=F_y$. Ddamulira, Luca and Rakotomalala \cite{DLR} considered first the akin problem with two terms on one side. They gave all Fibonacci numbers which are products of two Pell
numbers, and all Pell numbers which are products of two Fibonacci numbers. Hence the equations $F_z=P_xP_y$, and $P_z=F_xF_y$ were completely solved. In this paper, we allow products with two terms on both sides, where the two binary recurrences are representatives from the same class of sequences, which is a novel feature. The technique used in our proof is a variant of the combination of Baker's method and reduction procedures like LLL-algorithm, and a generalization of a result of Baker and Davenport by Dujella and Peth\H o \cite{DP}. We mention that a similar approach should work for equations of the same type involving products of more terms. However, this will certainly increase the amount of necessary computations. The principal result is recorded in the following

\begin{theorem}\label{t}
If equation (\ref{main}) holds for the positive integers $x_1\le x_2$ and $y_1\le y_2$, then
$$
x_1=1,\; x_2=2,\; y_1=1,\; y_2=1.
$$
\end{theorem}

We note that the method of Bert\'ok and Hajdu described in \cite{BH} probably also helps to solve (\ref{main}). This was confirmed by Bert\'ok via a personal communication.

\section{Preliminaries}

Here we list a few results which will be necessary later. Put $\lambda=\log 5/\log 3$.

\begin{lemma}\label{l1}
If $a\ge3$ is a real number and $x_1,x_2$ are positive integers, then 
$$a^{x_1+x_2-1}<(a^{x_1}-1)(a^{x_2}-1)<a^{x_1+x_2}.$$
\end{lemma}
\begin{proof}
The second inequality is obvious. The first one follows from 
$$
1-\frac{1}{a^{x_1}}-\frac{1}{a^{x_2}}+\frac{1}{a^{x_1+x_2}}>1-\frac{1}{a^{x_1}}-\frac{1}{a^{x_2}}\ge 1-\frac{1}{3^{x_1}}-\frac{1}{3^{x_2}}\ge \frac{1}{3}\ge\frac{1}{a}. 
$$
\end{proof}

\begin{corol}\label{c1}
Assume that the positive integers $x_1,x_2,y_1$ and $y_2$ satisfy (\ref{main}). Then
\begin{itemize}
	\item $(x_1+x_2-1)\log 3<(y_1+y_2)\log 5$,
	\item $(y_1+y_2-1)\log 5<(x_1+x_2)\log 3$.
\end{itemize}
\end{corol}
\begin{proof}
Apply Lemma \ref{l1} and (\ref{main}).
\end{proof}

\begin{lemma}\label{l2}
If $x_1+x_2\le 3$ or $y_1+y_2\le 3$ holds with the positive integers $x_1\le x_2$ and $y_1\le y_2$, then (\ref{main}) possesses only the solution $(x_1,x_2,y_1,y_2)=(1,2,1,1)$. 
\end{lemma}
\begin{proof}
The statement easily follows by directly checking all four possible cases.
\end{proof}

\begin{lemma}\label{l3}
Assume that (\ref{main}) holds, moreover $y_1+y_2\ge4$. Then $y_1+y_2<x_1+x_2$.
\end{lemma}
\begin{proof}
A short calculation admits $y_1+y_2<\lambda(y_1+y_2-1)$. Now the second statement of Corollary \ref{c1} proves the lemma.
\end{proof}

\begin{lemma}\label{l4}
Equation (\ref{main}) implies $2\nmid y_i$, $4\nmid x_i$, where $i\in\{1,2\}$.
\end{lemma}
\begin{proof}
Consider (\ref{main}) modulo 3, and 5, respectively.
\end{proof}

\begin{lemma}\label{l5}
If the real numbers $x$ and $K$ satisfy $|e^x-1|<K<3/4$, then $|x|<2K$.
\end{lemma}
\begin{proof}
{The assertion can be easily checked.}
\end{proof}

We need the following theorem from the theory of lower bounds on linear forms in logarithms of algebraic numbers. Recall Theorem 9.4 of \cite{BMS}, which is a modified version of a result of Matveev \cite{M}. Let $\L$ be an algebraic number field of degree $d_{\L}$ and let $\eta_1, \eta_2, \ldots, \eta_l \in \L$  not $0$
or $1$ and $d_1, \ldots, d_l$ be nonzero integers. Put
$$
D =\max\{|d_1|, \ldots, |d_l|, 3\}\qquad{\rm and}\qquad
\Gamma = \prod_{i=1}^l \eta_i^{d_i}-1.
$$
Let $A_1, \ldots, A_l$ be positive integers such that 
$$
A_j \geq h'(\eta_j) := \max \{d_{\L}h(\eta_j), |\log \eta_j|, 0.16\}, \quad {\text{\rm for}}\quad j=1,\ldots l,
$$
where for an algebraic number $\eta$ with minimal polynomial
$$
f(X)=a_0(X-\eta^{(1)})\cdots(X-\eta^{(u)})\in {\mathbb Z}[X]
$$
with positive $a_0$, we write $h(\eta)$ for its Weil height given by
$$
h(\eta)=\frac{1}{u}\left(\log a_0+\sum_{j=1}^u \max\{0,\log |\eta^{(j)}|\}\right).
$$
\begin{lemma} \label{Matveev}
 If $\Gamma \neq 0$ and $L \subseteq R $, then
\begin{equation*}
\label{ineq:matveev} \log |\Gamma| > -1.4 \cdot
30^{l+3}l^{4.5}d_{\L}^2(1+ \log d_{\L})(1+ \log D)A_1A_2\cdots A_l.
\end{equation*}
\end{lemma}

We also refer to the Baker-Davenport reduction method (see \cite[Lemma 5a]{DP}), which will be useful to reduce the bounds arising at the application of Lemma \ref{Matveev}.

\begin{lemma}\label{BD}
Let $\kappa\ne 0$ and $\mu$ be real numbers. Assume that $M$ is a positive integer. Let $P/Q$ be a convergent
of the continued fraction expansion of $\kappa$ such that $Q > 6M$,
and put
$$
\xi=\left\| \mu Q \right\| - M \cdot \| \kappa Q\|,
$$
where $\parallel\cdot\parallel$ denotes the distance from the
nearest integer. If $\xi>0$, then there is no solution of the
inequality
$$0 < |m\kappa -n+\mu| < AB^{-k}$$
in positive integers $m$, $n$ and $k$ with
$$
\frac{\log\left(AQ/\xi\right)}{\log B}\leq k \qquad\ and\qquad m\le M.
$$
\end{lemma}

\section{Proof of Theorem \ref{t}}

\subsection{The first bound}\label{ss1}

Recall that $\lambda=\log 5/\log 3$. Suppose that the positive integers $x_1\le x_2$ and $y_1\le y_2$ satisfy (\ref{main}). 
According to Lemma \ref{l2} we may assume $x_1+x_2\ge4$ and $y_1+y_2\ge4$.
Then
$$
\left|\frac{3^{x_1+x_2}}{5^{y_1+y_2}}-1\right|=\left|\frac{3^{x_1}+3^{x_2}-5^{y_1}-5^{y_2}}{5^{y_1+y_2}}\right|<
\frac{4\max\{3^{x_2},5^{y_2}\}}{5^{y_1+y_2}},
$$
and denote this upper bound by $B_1$. 
{The assumption}
$\max\{3^{x_2},5^{y_2}\}=5^{y_2}$ leads immediately to
$B_1=4/5^{y_1}$. Contrary, if $\max\{3^{x_2},5^{y_2}\}=3^{x_2}$, then
$$
B_1=\frac{4\cdot3^{x_2}}{5^{y_1+y_2}}=\frac{4}{5^{y_1+y_2-x_2/\lambda}}<\frac{4}{5^{(x_1+x_2-1)/\lambda-x_2/\lambda}}=\frac{4\cdot5^{1/\lambda}}{5^{x_1/\lambda}}=\frac{12}{5^{x_1/\lambda}},
$$
where the inequality follows from Corollary \ref{c1}. Thus we conclude
\begin{equation}\label{e1}
\left|\frac{3^{x_1+x_2}}{5^{y_1+y_2}}-1\right|<\frac{12}{5^{\min\{x_1/\lambda,y_1\}}}.
\end{equation}
Let the term in the absolute value of left-hand side of (\ref{e1}) be denoted by $\Gamma_1$, which is obviously non-zero. Put $z=\max\{x_1+x_2,y_1+y_2\}$. Clearly, Lemma \ref{l3} gives  $z=x_1+x_2$. Now we apply Lemma \ref{Matveev} with $l=2$, $\eta_1=3$, $\eta_2=5$, $L=Q$, $D=z$, $A_1=\log 3$, $A_2=\log 5$. This provides
$$\log |\Gamma_1|>-1.4\cdot 10^9Z,$$ 
where $Z=(1+\log z)$, and then together with (\ref{e1}) we obtain
$$
\min\{x_1/\lambda,y_1\}<8.7\cdot10^8Z. 
$$
For the next calculations we distinguish two cases.

{\it Case 1.} Assume $3^{x_1}<5^{y_1}$, or equivalently $x_1/\lambda<y_1$. Equation (1) leads to
\begin{equation}\label{e2}
\left|\Gamma_{2,1}\right|:=\left|\frac{(3^{x_1}-1)3^{x_2}}{5^{y_1+y_2}}-1\right|=\left|\frac{3^{x_1}-5^{y_1}-5^{y_2}}{5^{y_1+y_2}}\right|<
\frac{3^{x_1}+2\cdot5^{y_2}}{5^{y_1+y_2}}=\frac{\frac{3^{x_1}}{5^{y_2}}+2}{5^{y_1}}<\frac{3}{5^{y_1}}.
\end{equation}
In order to use Lemma \ref{Matveev} again, now for $|\Gamma_{2,1}|$ we specify $l=3$, $\eta_1=3^{x_1}-1$, $\eta_2=3$, $\eta_3=5$, $L=Q$. Moreover fix
$D=\max\{1,x_2,y_1+y_2\}<z$, $A_2=\log 3$, $A_3=\log 5$. Finally, 
$$
h(3^{x_1}-1)<(\log 3)x_1=(\log 5)x_1/\lambda<(\log 5)\cdot8.7\cdot10^8Z<1.5\cdot10^9Z=A_1.
$$

{\it Case 2.} Now let $3^{x_1}>5^{y_1}$. Hence $x_1/\lambda>y_1$. Equation (1) implies
\begin{equation}\label{e3}
\left|\Gamma_{2,2}\right|:=\left|\frac{(5^{y_1}-1)5^{y_2}}{3^{x_1+x_2}}-1\right|=\left|\frac{5^{y_1}-3^{x_1}-3^{x_2}}{3^{x_1+x_2}}\right|<
\frac{5^{y_1}+2\cdot3^{x_2}}{3^{x_1+x_2}}=\frac{\frac{5^{y_1}}{3^{x_2}}+2}{3^{x_1}}<\frac{3}{3^{x_1}}.
\end{equation}
For $|\Gamma_{2,2}|$ we have $l=3$, $\eta_1=5^{y_1}-1$, $\eta_2=5$, $\eta_3=3$, $L=Q$,
$D=\max\{1,y_2,x_1+x_2\}=z$, $A_2=\log 5$, $A_3=\log 3$, and 
$h(5^{y_1}-1)<(\log 5)y_1<A_1$.

Observe that we are able to apply Lemma \ref{Matveev} simultanously for $|\Gamma_{2,1}|$ and $|\Gamma_{2,2}|$ since up to the order we have the same parameters. It gives
$$\log |\Gamma_{2,i}|>-3.798\cdot 10^{20}Z^2,\qquad (i=1,2)$$ 
consequently by (\ref{e2}), and (\ref{e3}), respectively we derive
\begin{equation}\label{isti1}
(\log 5)y_1-\log 3<(\log 5)y_1<h_1Z^2\qquad{\rm and}\qquad (\log 3)x_1-\log 3<(\log 3)x_1<h_1Z^2,
\end{equation}
where $h_1=3.8\cdot 10^{20}$.

Return again to the conditions of the separation of Cases 1 and 2 for a while.

{\it Case 1.} ($3^{x_1}<5^{y_1}$, $x_1/\lambda<y_1$.) Equation (1) also leads to
\begin{equation}\label{e4}
\left|\Gamma_{3,1}\right|:=\left|\frac{(3^{x_1}-1)3^{x_2}}{(5^{y_1}-1)5^{y_2}}-1\right|=\left|\frac{3^{x_1}-5^{y_1}}{(5^{y_1}-1)5^{y_2}}\right|<
\frac{3^{x_1}+5^{y_1}}{(5^{y_1}-1)5^{y_2}}=\frac{\frac{3^{x_1}}{5^{y_1}}+1}{\left(1-\frac{1}{5^{y_1}}\right)5^{y_2}}<\frac{3}{5^{y_2}}.
\end{equation}
Preparing the application of Lemma \ref{Matveev}, for $|\Gamma_{3,1}|$ we fix $l=4$, $\eta_1=3^{x_1}-1$, $\eta_2=3$, $\eta_3=5^{y_1}-1$, $\eta_4=5$, $L=Q$. Furthermore
$D=\max\{1,x_2,1,y_2\}<z$, $A_1=1.5\cdot10^9Z$, $A_2=\log 3$, $A_4=\log 5$, and  
$h(5^{y_1}-1)<(\log 5)y_1<3.8\cdot 10^{20}Z^2=A_3$.
We will see soon, that Case 2 essentially admits the same parameters.

{\it Case 2.} ($3^{x_1}>5^{y_1}$, $x_1/\lambda>y_1$.) We derive from (1) that 
\begin{equation}\label{e5}
\left|\Gamma_{3,2}\right|:=\left|\frac{(5^{y_1}-1)5^{y_2}}{(3^{x_1}-1)3^{x_2}}-1\right|=\left|\frac{5^{y_1}-3^{x_1}}{(3^{x_1}-1)3^{y_2}}\right|<
\frac{5^{y_1}+3^{x_1}}{(3^{x_1}-1)3^{x_2}}=\frac{\frac{5^{y_1}}{3^{x_1}}+1}{\left(1-\frac{1}{3^{x_1}}\right)3^{x_2}}<\frac{3}{3^{x_2}}.
\end{equation}
To bound $|\Gamma_{3,2}|$ from below let $l=4$, $\eta_1=5^{y_1}-1$, $\eta_2=5$, $\eta_3=3^{x_1}-1$, $\eta_4=3$, $L=Q$,
$D=\max\{1,y_1,1,x_2\}<z$, $A_1=1.5\cdot10^9Z$, $A_2=\log 5$, $A_4=\log 3$, and  
$h(3^{x_1}-1)<(\log 3)x_1<3.8\cdot 10^{20}Z^2=A_3$.

Thus, Lemma \ref{Matveev} returns with
$$\log |\Gamma_{3,i}|>-1.58\cdot 10^{43}Z^4,\qquad(i=1,2),$$ 
and finally with $h_2=1.6\cdot 10^{43}$ we have
\begin{equation}\label{isti2}
(\log 5)y_2-\log 3<(\log 5)y_2<h_2Z^4\qquad{\rm and}\qquad (\log 3)x_2-\log 3<(\log 3)x_2<h_2Z^4.
\end{equation}

Now we combine (\ref{isti1}) and (\ref{isti2}) to bound $y_1+y_2$ and $x_1+x_2$. Recall that $z=\max\{x_1+x_2,y_1+y_2\}=x_1+x_2$, $Z=1+\log z$. The right-hand sides yield
$$
(\log 3)(x_1+x_2)-2\log 3<h_1Z^2+h_2Z^4.
$$
In case of the left-hand sides together with Corollary \ref{c1} we find
$$
(\log 3)(x_1+x_2-1)-2\log 3<(\log 5)(y_1+y_2)-2\log 3<h_1Z^2+h_2Z^4. 
$$
Both inequailities provide upper bounds on $z$, the result is recorded in the following
\begin{prop}\label{th1}
Assume that $z=\max\{x_1+x_2,y_1+y_2\}\ge4$. Then for the solutions of (\ref{main}) 
\begin{equation}\label{z}
z=x_1+x_2<3\cdot10^{51}
\end{equation}
holds.
\end{prop}

Note that 
{in Proposition \ref{th1} we have $z=x_1+x_2$ since  $x_1+x_2>y_1+y_2$ is obvious from Lemma \ref{l3}. Moreover
}
$y_1+y_2<2.1\cdot10^{51}$ follows from Proposition \ref{th1} and Corollary \ref{c1}.

\subsection{The second bound}\label{ss2}

In the sequel we assume $x_1\ge3$ and $y_1\ge2$. The remaining cases where $x_1<3$ or $y_1<2$ will be handled later, in Subsection \ref{rc}.

Now $\min\{x_1/\lambda,y_1\}\ge2$, hence the right-hand side of (\ref{e1}) does not exceed $12/25<3/4$. Consequently, Lemma \ref{l5} with $x=(x_1+x_2)\log 3-(y_1+y_2)\log 5$ implies
\begin{equation}\label{gamma4}
|\Gamma_4|:=\left|(x_1+x_2)\log 3-(y_1+y_2)\log 5\right|<\frac{24}{5^{\min\{x_1/\lambda,y_1\}}}.
\end{equation}
For the computational aspects of the application of LLL algorithm we refer to the book of H.~Cohen \cite{C}, page 58-63, moreover the {\tt LLL(lvect, integer)} command of the package IntegerRelation in Maple. We implemented the computations in Maple by following Cohen's approach. 
The LLL-algorithm (with $x_1+x_2<3\cdot 10^{51}$, $y_1+y_2<2.1\cdot 10^{51}$) provides
$$
4.5\cdot10^{-56}<|\Gamma_4|,
$$
and combining it with (\ref{gamma4}) we derive
$\min\{x_1/\lambda,y_1\}<81.2$.
Thus we have the following
\begin{prop}\label{th2}
Assume that 
$x_1\ge3$, $y_1\ge2$. If (\ref{main}) holds, then
\begin{equation*}\label{m}
\min\{x_1/\lambda,y_1\}\le81.2.
\end{equation*}
\end{prop}

\subsection{The third bounds}\label{ss3}

Recall (\ref{e2}) and (\ref{e3}), respectively, in accordance with the two cases of Subsection  \ref{ss1}.

{\it Case 1.} ($3^{x_1}<5^{y_1}$.) Since $3/5^{y_1}<3/4$, by Lemma \ref{l5} we obtain
$$
\left|\frac{1}{\lambda}x_2-(y_1+y_2)+\frac{\log(3^{x_1}-1)}{\log 5}\right|<\frac{6}{5^{y_1}\log 5}<3.8\cdot 5^{-y_1},
$$
where $x_2<3\cdot 10^{51}$, and $3\le x_1<82.9\cdot\lambda<118.96$. Apply the Baker-Davenport type reduction method described in Lemma \ref{BD} together with the parameters $M=3\cdot 10^{51}$, $A=3.8$, $B=5$, $m=x_2$, $\kappa=1/\lambda$, $n=y_1+y_2$, $\mu=\log(3^{x_1}-1)/\log(5)$ with $3\le x_1\le 118$, $4\nmid x_1$ (87 cases). Note that 
$$Q_{113}=49979470671933915311803624529695074923111987539096229\approx5\cdot10^{52}$$ 
is the first denominator exceeding $6M$. For the possible values of $x_1$, in all cases we found $y_1\le82$.

{\it Case 2.} ($3^{x_1}>5^{y_1}$.) If $x_1\ge2$, then $3/3^{x_1}<3/4$, and Lemma \ref{l5} admits
$$
\left|\lambda y_2-(x_1+x_2)+\frac{\log(5^{y_1}-1)}{\log 3}\right|<\frac{6}{3^{x_1}\log 3}<5.5\cdot 3^{-x_1}.
$$
Here $y_2<2.1\cdot 10^{51}$, and $2\le y_1\le81$, $y_1$ is odd (40 possibility for $y_1$). Similarly to Case 1, we use Lemma \ref{BD}, which leads to $x_1\le115$.

We summarize the last computational results as follows.
\begin{prop}\label{th3}
Assume again that 
$x_1\ge3$, $y_1\ge2$. Then from equation (\ref{main}) we conclude
\begin{equation*}\label{ma}
x_1\le 118\qquad{\rm and}\qquad y_1\le81.
\end{equation*}
\end{prop}

\subsection{The final bounds, and verification}\label{ss4}

In (\ref{e4}), and (\ref{e5}) we found that $|\Gamma_{3,1}|<3/5^{y_2}$, and $|\Gamma_{3,2}|<3/3^{x_2}$, respectively. Supposing $y_2\ge2$, and $x_2\ge3$, it is obvious that both $|\Gamma_{3,i}|<4/5$, subsequently we can apply Lemma \ref{l5}. It gives
\begin{equation}\label{v1}
\left|(\log 3)x_2-(\log 5)y_2+\log\left(\frac{3^{x_1}-1}{5^{y_1}-1}\right)\right|<\frac{6}{5^{y_2}}
\end{equation}
in Case 1. Knowing $x_2<3\cdot 10^{51}$, $y_2<2.1\cdot 10^{51}$ we use the LLL algorithm for each possible pair $(x_1,y_1)$ with the bounds given in Proposition \ref{th3}. To reduce the time of the calculations we also exploit that the conditions $4\nmid x_1$, $2\nmid y_1$, and $3^{x_1}<5^{y_2}$ (Case 1) also hold. The procedure yields lower bound $K_{x_1,y_1}$ for the left-hand side of (\ref{v1}) in each case, and comparing it with $6/5^{y_2}$ we obtain an upper bound on $y_2$. The maximum of the upper bounds is 159. Finally,
$$
3^{x_2}<1+\frac{5^{y_1+y_2}}{3^{x_1}-1}\le 1+\frac{5^{81+159}}{3^{3}-1}
$$
gives $x_2\le 348$. The divisibility condition reduces it to $x_2\le 347$.

A similar treatment works for 
\begin{equation*}\label{v2}
\left|(\log 5)y_2-(\log 3)x_2+\log\left(\frac{5^{y_1}-1}{3^{x_1}-1}\right)\right|<\frac{6}{3^{x_2}}
\end{equation*}
in Case 2. It returns with $x_2\le235$, and then $y_2\le 238$. Thus $y_2\le 237$.
\begin{prop}\label{th4}
Assume again that 
$x_1\ge3$, $y_1\ge2$. Equation (\ref{main}) can hold only if
\begin{equation*}\label{maa}
x_2\le 347\qquad{\rm and}\qquad y_2\le237.
\end{equation*}
\end{prop}

Now a simple computer search shows that there is no solution to equation (\ref{main}) if
$3\le x_1\le118$, $2\le y_1\le 81$, $x_1\le x_2\le 347$, and $y_1\le y_2\le 237$.

\subsection{The remaining cases}\label{rc}

This subsection is devoted to handle separately the cases $x_1=2$, $x_1=1$, and $y_1=1$. 

{\it Case $x_1=2$.} Now equation (\ref{main}) has the form
\begin{equation}\label{rceq1}
8\cdot(3^{x_2}-1)-(5^{y_1}-1)(5^{y_2}-1)=0,
\end{equation}
where the exponents $x_2$, $y_1$, and $y_2$ are positive integers, $x_2\ge2$. Let $m=3^2\cdot 7\cdot 13=819$. Observe that the Carmichael function has the small value $\lambda(m)=12$. Since
$3^x\equiv 3^{x+6}\;(\bmod\,m)$ holds if $x\ge2$, and $5^{12}\equiv 1\;(\bmod\,m)$ fulfils, we checked the possibilities $2\le x_2\le 7$, $1\le y_1,y_2\le12$ for the left-hand side of (\ref{rceq1}) modulo $m$,
and it never gave 0. Hence there is no solution to (\ref{main}) with $x_1=2$.

{\it Case $x_1=1$.} Equation (\ref{main}) gives 
\begin{equation}\label{rceq2}
2\cdot(3^{x_2}-1)=(5^{y_1}-1)(5^{y_2}-1),
\end{equation}
the exponents $x_2$, $y_1$, and $y_2$ are positive integers. 
Since $4\le y_1+y_2$, by Lemma \ref{l3} we may assume $x_2\ge y_1+y_2$.

What follows is very similar to the treatment of the Subsections \ref{ss1} and \ref{ss2} therefore here we give less details. From (\ref{rceq2}) we have
\begin{equation}\label{u1}
\left|\Gamma_6\right|:=\left|\frac{2\cdot3^{x_2}}{5^{y_1+y_2}}-1\right|=\left|\frac{3-5^{y_1}-5^{y_2}}{5^{y_1+y_2}}\right|<
\frac{2\cdot 5^{y_2}}{5^{y_1+y_2}}=\frac{2}{5^{y_1}}.
\end{equation}
Apply again Lemma \ref{Matveev}, clearly with $D=\max\{x_2,y_1+y_2\}=x_2$. It provides immediately
$\log\left|\Gamma_6\right|>-1.755\cdot10^{11}(1+\log x_2)$, which together with (\ref{u1}) entails
\begin{equation}\label{y11}
y_1<1.1\cdot10^{11}(1+\log x_2)=:K_1(x_2).
\end{equation}
From equation (\ref{rceq2}) we can also conclude
\begin{equation}\label{u2}
0<\Gamma_7:=\frac{(5^{y_1}-1)5^{y_2}}{2\cdot3^{x_2}}-1=\frac{5^{y_1}-3}{2\cdot3^{x_2}}<
\frac{5^{y_1}}{2\cdot3^{x_2}}<\frac{5}{5^{y_2}}.
\end{equation}
The last inequality is a consequence of Lemma \ref{l1} and (\ref{rceq2}):
$$
5^{y_1+y_2-1}<(5^{y_1}-1)(5^{y_2}-1)=2\cdot(3^{x_2}-1)<2\cdot3^{x_2}.
$$
 Using the theorem of Matveev (Lemma \ref{Matveev}) for $\Gamma_7$, it returns with $\log\Gamma_7>-3.459\cdot10^{24}(1+\log x_2)^2$. Hence, by  
(\ref{u2})
$$
y_2<2.16\cdot10^{24}(1+\log x_2)^2=:K_2(x_2)
$$
follows. Now
$$
\frac{\log3}{\log5}\,x_2=\frac{\log3}{\log5}(x_1+x_2-1)<y_1+y_2<K_1(x_2)+K_2(x_2)
$$
leads to the absolute bound
\begin{equation*}\label{y12}
x_2<1.4\cdot10^{28}.
\end{equation*}

Clearly, (\ref{u1}) implies $\left|\Gamma_6\right|<3/4$. Thus Lemma \ref{l5} yields
\begin{equation*}\label{LLLu1}
\left|x_2\log3-(y_1+y_2)\log5+\log2\right|<\frac{4}{5^{y_1}}.
\end{equation*}
The application of the LLL-algorithm with the bound $y_1+y_2<x_2<1.4\cdot10^{28}$ leads to  
$$
y_1\le93=:K_1^\star.
$$
Now we repeat the treatment from (\ref{y11}), replacing $K_1(x_2)$ by $K_1^\star$. Lemma \ref{Matveev} provides
$$
y_2<3.4\cdot10^{14}(1+\log x_2)=:K_2^\star(x_2).
$$
Henceforward
$$
\frac{\log3}{\log5}\,x_2<K_1^\star+K_2^\star(x_2),
$$
and then 
\begin{equation*}\label{y13}
y_2<x_2<1.92\cdot10^{16}.
\end{equation*}

The last step of this specific case is to exploit (\ref{u2}). Clearly, $5/5^{y_2}<3/4$, subsequently
\begin{equation*}\label{BDu1}
\left|y_2\frac{\log5}{\log3}-x_2+\frac{\log((5^{y_1}-1)/2)}{\log3}\right|<\frac{10}{5^{y_2}\log3}<\frac{10}{5^{y_2}}.
\end{equation*}
We used the Baked-Davenport type reduction method (Lemma \ref{BD}) for all the possible cases
$y_1=3,5,\dots,93$ (46 values) and found $y_2\le29$. Thus $y_1\le29$, and a verification of (\ref{rceq2}) with finitely many integers on its right-hand side gives no solution to (\ref{rceq2}).

{\it Case $y_1=1$.} Now equation (\ref{e1}) returns with 
\begin{equation}\label{rceq3}
(3^{x_1}-1)(3^{x_2}-1)=4\cdot(5^{y_2}-1).
\end{equation}
A complete analogue of the treatment of the {\it Case $x_1=1$} can be applied to solve (\ref{rceq3}). Here we omit the details, and declare again that no solution exists unless $y_1=y_2=1$.

\bigskip

{\bf Acknowledgements.} The authors thank to the referee and F.~Luca for the valuable remarks. We are also grateful to Cs. Bert\'ok for ana\-lyzing the problem with the method described in \cite{BH}.


\begin{thebibliography}{0}

\bibitem{BH} Cs. Bert\'ok, L. Hajdu, \textit{On a Hasse-type principle for exponential diophantine
equations and its applications}, Math. Comp. {\bf 85} (2016), 849--860.

\bibitem{BHLS} Cs. Bert\'ok, L. Hajdu, F. Luca, D. Sharma, \textit{
On the number of non-zero digits of integers in multi-base representations}, Publ. Math. Debrecen {\bf 90} (2017), 181--194.

\bibitem{BMS} Y. Bugeaud, M. Maurice, and S. Siksek,  \textit{Classical and modular approaches to exponential Diophantine equations I. Fibonacci and Lucas perfect powers}, Annals Math. \textbf{163} (2006), 969--1018.

\bibitem{C} H. Cohen, \textit{Number Theory I: Tools and Diophantine Equations}, Graduate Texts in Mathematics 239, Springer, 2007.

\bibitem{DLR} M. Ddamulira, F. Luca, and M. Rakotomalala, \textit{Fibonacci numbers which are products of two Pell numbers}, Fibonacci Quart. \textbf{54}, 11--18.

\bibitem{DP} A. Dujella and A. {Peth\H o}, \textit{A generalization of a theorem of Baker and Davenport}, Quart. J. Math. Oxford Ser. (2) \textbf{49} (1998), 291--306.

\bibitem{M} E. M. Matveev, \textit{An explicit lower bound for a homogeneous rational linear form in logarithms of algebraic numbers, II}, Izv. Ross. Akad. Nauk Ser. Mat. \textbf{64} (2000), 125--180. English translation in Izv. Math. \textbf{64} (2000), 1217--1269.

\bibitem{S} C. L. Stewart, \textit{On the representation of an integer in two different bases}, Journal f\"ur die reine und angewandte Mathematik \textbf{319} (1980), 63--72. 

\bibitem{SS1} H. P. Schlickewei, W. M. Schmidt, \textit{Linear equations in members of recurrence sequaneces}, Ann. Scuola Norm. Sup. Pisa Cl. Sci. \textbf{20} (1993), 219--246. 

\bibitem{SS} H. G. Senge, E. G. Strauss, \textit{P. V. numbers and the sets of multiplicity}, Periodica Math. Hung. \textbf{3} (1973), 93--100.

\end{thebibliography}
\end{document}